\documentclass[12pt,reqno]{amsart}
\usepackage{cite}
\usepackage[colorlinks,citecolor=blue,urlcolor=blue,hypertexnames=false]{hyperref}
\usepackage{amssymb,amsmath,amsfonts,latexsym}
\usepackage{bm}
\usepackage{mathtools, mathrsfs}
\usepackage{enumerate}

\usepackage{array,graphics,color, csquotes}
\allowdisplaybreaks

\numberwithin{equation}{section}

\newtheorem{dfn}{Definition}[section]
\newtheorem{thm}[dfn]{Theorem}
\newtheorem{lma}[dfn]{Lemma}

\newtheorem{prob}[dfn]{Problem}

\DeclarePairedDelimiterX{\norm}[1]{\lVert}{\rVert}{#1}
\DeclarePairedDelimiterX{\bnorm}[1]{\big\lVert}{\big\rVert}{#1}
\DeclarePairedDelimiterX{\Bnorm}[1]{\Big\lVert}{\Big\rVert}{#1}

\newcommand{\R}{\mathbb{R}}
\newcommand{\C}{\mathbb{C}}

\newcommand{\N}{\mathbb{N}}

\newcommand{\cS}{\mathcal{S}}
\newcommand{\hil}{\mathcal{H}}

\newcommand{\bh}{\mathcal{B}(\hil)}

\newcommand{\Tr}{\operatorname{Tr}}

\newcommand{\ie}{\hookrightarrow}
\newcommand{\nie}{\not\hookrightarrow}

\newcommand{\la}{\langle}
\newcommand{\ra}{\rangle}

\begin{document}
	\title{Isometric Embeddability of Schatten Classes Revisited}

	\author[Chattopadhyay] {Arup Chattopadhyay}
	\address{Department of Mathematics, Indian Institute of Technology Guwahati, Guwahati, 781039, India}
	\email{arupchatt@iitg.ac.in, 2003arupchattopadhyay@gmail.com}
	
	\author[Pradhan]{Chandan Pradhan}
	\address{Department of Mathematics and Statistics, University of New Mexico, 311 Terrace Street NE, Albuquerque, NM 87106, USA}
	\email{chandan.pradhan2108@gmail.com, cpradhan@unm.edu}
	
	\author[Skripka] {Anna Skripka}
	\address{Department of Mathematics and Statistics, University of New Mexico, 311 Terrace Street NE, Albuquerque, NM 87106, USA}
	\email{skripka@math.unm.edu}

	
	
	
	\begin{abstract}
 In this note, we summarize known results and open questions on the existence of isometric embeddings between different Schatten classes. We also provide a brief overview of the relevant methods.
	\end{abstract}
	
	\maketitle
	
	\section{Introduction}
Given two (quasi-)Banach spaces, a fundamental problem is to determine whether one can be isometrically embedded into the other. We recall that an isometric embedding of a (quasi-)Banach space $X$ into a (quasi-)Banach space $Y$ is a linear isometry mapping $X$ to $Y$. The goal of this note is to survey results and open questions on the existence of isometric embeddings between different Schatten classes on both finite- and infinite-dimensional Hilbert spaces. We also discuss the embeddability of sequence spaces that arise as isometric copies of commutative subspaces of Schatten classes and of some relevant function spaces.

Let $\hil$ be a complex separable Hilbert space and let $\bh$ denote the algebra of all bounded linear operators on $\hil$. For $0<p<\infty$, the Schatten class $\cS_p(\hil)$ is defined as the set of all compact operators $T\in\bh$ satisfying
\[
\|T\|_{p}
:=
\Big(\sum_{n=1}^\infty s_n(T)^p\Big)^{1/p}
=
\left( \mathrm{Tr}(|T|^p)\right)^{1/p}
<\infty,
\]
where $(s_n(T))_{n\ge1}$ denotes the sequence of singular values of $T$, that is, the eigenvalues of $|T| := (T^*T)^{1/2}$, counted with multiplicity. For $p=\infty$, we set $\cS_\infty(\hil)$ to be the space of all compact operators on $\hil$, equipped with the operator norm $\|\cdot\|$. For $1\le p\leq \infty$, $\cS_p(\hil)$ is a Banach space with respect to the norm $\|\cdot\|_p$, whereas for $0<p<1$, $\cS_p(\hil)$ is a quasi-Banach space with respect to the quasi-norm $\|\cdot\|_p$. When $\hil=\mathbb{C}^n$, we write $\cS_p^n := \cS_p(\mathbb{C}^n)$.

If $\dim(\hil)=\infty$, there is a natural isometric embedding of the sequence spaces
\begin{align*}
&\ell_p(\C):=\big\{ (x_n) :\; x_n\in\C,\; \|(x_n)\|_p^p:=\sum_{n=1}^\infty |x_n|^p < \infty \big\},\quad p\in(0,\infty)
\end{align*}
into $\cS_p(\hil)$ for $p\in(0,\infty)$ and $$ c_0(\C) := \big\{ (x_n) :\;  x_n\in\C,\; \lim\limits_{n\to \infty} x_n=0, \; \|(x_n)\|_\infty:=\sup_n|x_n| \big\}$$ into $\cS_\infty(\hil)$, as closed subspaces. Indeed, each $\ell_p(\C)$ (respectively, $c_0(\C)$) identifies isometrically with the subspace of operators in $\cS_p(\hil)$ (respectively, $S_\infty(\hil)$) that are diagonal with respect to a fixed orthonormal basis of $\hil$. Hence, the classical sequence spaces appear as canonical commutative subspaces of the noncommutative Schatten classes and are therefore relevant to the embeddability problem for the latter.

Another class of commutative Banach spaces relevant to the problem of embeddability of Schatten classes is provided by the classical of measurable functions
\begin{align*}
&L^p(\Omega, \C, d\mu) := \big\{ f:\Omega \to \C :\; \|f\|_p^p := \int_{\Omega} |f|^p \, d\mu < \infty \big\}, \quad p\in[1,\infty),\\
&L^\infty(\Omega, \C, d\mu) := \big\{ f:\Omega \to \C :\; \|f\|_\infty := \operatorname{ess\,sup}_{\Omega} |f| < \infty \big\},
\end{align*}
where $(\Omega, d\mu)$ is a measure space.
 The connection between the embeddability of $L^p(\Omega, \mathbb{C}, d\mu)$ and that of $\cS_p(\hil)$ is fundamental to our new result (see Theorem \ref{mainthm}) and is discussed in the next section.

 The existence of isometric embeddings between finite-dimensional $\ell_p$ spaces has been extensively studied, and in most cases the answer turns out to be negative. Let $\mathbb K$ be one of the fields $\R, \C, \mathbb H$, where $\mathbb H$ denotes the division ring of quaternions $\xi = a + b\bm i + c\bm j + d\bm k$ with real coefficients $a, b, c, d$. For $0<p<\infty$, we denote by $\ell_p^n(\mathbb{K})$ the $\mathbb{K}$-linear space $\mathbb{K}^n$ equipped with the $\ell_p$ quasi-norm
\[\|(a_1,\ldots, a_n)\|_p=\Big(\,\sum_{k=1}^{n}|a_k|^p\,\Big)^{\frac{1}{p}},\]
and  $\ell_\infty^n(\mathbb{K})$ by the $\mathbb{K}$-linear space $\mathbb{K}^n$ equipped with the sup-norm	
\[\|(a_1,\ldots, a_n)\|_\infty=\max_{1\leq k\leq n}\, |a_k|.\]

In the study of isometries between Schatten $p$-classes, or between their canonical subspaces $\ell_p^n(\C)$, several embedding problems arise naturally.

Throughout this note, we denote the existence of an isometric embedding of $X$ into $Y$ by $X\ie Y$ and its absence by $X\nie Y$.

\begin{prob}\label{prob-0}
	Let $(q,p)\in(0,\infty]\times(0,\infty]$ with $p\neq q$.
	\begin{enumerate}
		\item Let $m,n\in\N$ with $2\le m\le n<\infty$. When does $\ell_q^m(\mathbb K)\ie\ell_p^n(\mathbb K)$?
		\item Let $m,n\in\N$ with $2\le m\le n<\infty$. When does $\cS_q^m\ie\cS_p^n$?
		\item Let $\hil$ be a complex separable Hilbert space. When does $\cS_q(\hil)\ie\cS_p(\hil)$?
	\end{enumerate}
\end{prob}

The partial solutions and the unsettled cases of Problem \ref{prob-0} are discussed in the next two sections.

The embeddability of Schatten classes is a particular case of the problem of embeddability of noncommutative $L_p$-spaces. The latter, more general problem is beyond the scope of this paper, but interested readers can find a discussion in
\cite{Ju00, JuPa10, JuPa08, JuPa082, Ra08, SuXu03, Xu06, Ye81}.

\section{Results on embeddability.}

\subsection{Embeddability of sequence spaces.}

In his seminal work \cite{Ba32}, S.~Banach characterized all linear isometries of $\ell_p(\C)$ onto itself, as stated in the theorem below.

\begin{thm}(\!\!\cite{Ba32})
	Let $U$ be a linear onto isometry on $\ell_p(\C)$, where $1\le p\neq 2$. Then, there exist a function $\phi:\N\to\N$ and a sequence $\{\epsilon_n\}_{n\in\N}$ such that
	\begin{enumerate}[(a)]
		\item $\phi$ is a permutation of $\N$,
		\item $|\epsilon_n|=1$ for all $n\in\N$,
		\item for every $(x_n)\in\ell_p(\C)$,
		\[
		U((x_n))=(\epsilon_n x_{\phi(n)}).
		\]
	\end{enumerate}
	Conversely, for any $\phi$ and $\{\epsilon_n\}$ satisfying {\rm(a)} and {\rm(b)}, the operator $U$ defined by {\rm(c)} is a linear isometry on $\ell_p(\C)$.
\end{thm}

The following canonical embeddings demonstrate the existence of affirmative solutions to special cases of Problem \ref{prob-0} (1).

\begin{thm}
Let $m,n\in\N$ satisfy $2 \leq m \leq n$. Then,
\begin{align}
\label{simple_cases}
&\ell_p^m(\mathbb K)\ie\ell_p^n(\mathbb K)\ie\ell_p(\mathbb K)\quad\text{for all }  p\in (0,\infty],\\
\label{iso-1-infty}&\ell_1^2(\R)\ie \ell_\infty^n(\R).
\end{align}
\end{thm}

\begin{proof}
All embeddings in \eqref{simple_cases} are immediate. An example of a linear isometry in \eqref{iso-1-infty} is given by $(x,y)\mapsto (x-y,\,x+y,\,0,\ldots,0)\in\R^n$.
\end{proof}

Moreover, given $m \in \mathbb{N}$ and $p \in 2\mathbb{N}$, 
upper bounds on $n \in \mathbb{N}$ 
for which $\ell_2^m(\mathbb K)\ie\ell_p^n(\mathbb K)$ are obtained in 
the next theorem using an equivalence between the isometric embeddings and the cubature formulas for polynomial functions on projective spaces.

	\begin{thm}(\!\!\cite[Theorem 2]{LySh04})
Let $m\in\N$ and $p\in 2\N$. Then there exists $n\in\N$ such that $m\le n\le \Lambda(m,p)$, where
		\begin{equation}\label{inteq4}
			\Lambda(m,p) =
\begin{cases*}
\begin{pmatrix} m+p-1 \\ m-1 \end{pmatrix}, & $\mathbb K=\mathbb{R}$,\\
\begin{pmatrix} m+p/2-1 \\ m-1 \end{pmatrix}^2, & $\mathbb K=\mathbb{C}$,\\
\frac{1}{2m-1}\begin{pmatrix} 2m+p/2-2 \\ 2m-2 \end{pmatrix}
\begin{pmatrix} 2m+p/2-1 \\ 2m-2 \end{pmatrix}, &$\mathbb K=\mathbb{H}$,
			\end{cases*}
		\end{equation}
such that $\ell_2^m(\mathbb K)\ie\ell_p^n(\mathbb K)$.
	\end{thm}
	
The origins of the embedding \eqref{inteq4} and the associated bounds on $n$ trace back to \cite{LyVa93}, where similar problems were resolved
	 in the setting of finite-dimensional $\ell_p$-spaces over $\mathbb{R}$. In the case $\mathbb K=\mathbb{C}$, the upper bound $\Lambda(m,p)$ in \eqref{inteq4} was previously obtained in \cite{Konig}.  We refer interested readers to \cite{Ly08,Ly09,LySh01, LySh04} for further details. The following result on non-embeddability is obtained in \cite{LySh04}.

	\begin{thm}(\!\!\cite[Theorem 1]{LySh04})\label{comm-thm-1}
Let $m,n\in\N$ satisfy $2 \leq m \leq n$. Then, the following assertions hold.
\begin{enumerate}
\item $\ell_q^m(\mathbb{K})\nie\ell_p^n(\mathbb{K})$\; for\; $\mathbb{K}\in\{\C, \mathbb{H}\}$\; and\;
$(q,p)\in [1,\infty]\times[1,\infty] \setminus \big(\{2\}\times 2\N\big)$,\; $q\neq p$.

\item $\ell_q^m(\mathbb{R})\nie\ell_p^n(\mathbb{R})$\; for\;
$(q,p)\in[1,\infty]\times[1,\infty]\setminus\big((\{2\}\times 2\N)\bigcup( \{1,\infty\}\times\{1,\infty\})\big)$,\; $q\neq p$.
\end{enumerate}
\end{thm}

The non-embeddability $\ell_1^2(\C)\nie \ell_\infty^n(\C)$ for $n\geq 2$ ensured by Theorem~\ref{comm-thm-1} contrasts with the embedding $\ell^2_1(\R)\ie\ell_\infty^n(\R)$ given by \eqref{iso-1-infty}.

\smallskip
More recently, the embeddability of $\ell_q^m(\mathbb K)$ into $\ell_p^n(\mathbb K)$ was  investigated in \cite{ChHoPrRa23} in the setting of quasi-Banach spaces for $(q,p)\in(0,\infty)\times(0,1)$.

\begin{thm}(\!\!\cite[Theorem 1.2]{ChHoPrRa23})\label{comm-thm-2}
Let $m,n\in\N$ satisfy $2\leq m\leq n<\infty$. Then, the following assertions hold.
\begin{enumerate}
\item $\ell_q^m(\mathbb{C})\nie\ell_p^n(\mathbb{C})$\; for\; $(q,p)\in (0,\infty)\setminus 2\mathbb{N}\times (0,1)$,\; $q\neq p$.
\item $\ell_q^m(\mathbb{R})\nie\ell_p^n(\mathbb{R})$\; for\; $(q,p)\in (0,\infty)\times (0,1)$,\; $q\neq p$.
\end{enumerate}
\end{thm}

\subsection{Embeddability of function spaces.}

In \cite{Ba32}, S. Banach characterized isometries of the function spaces $L_p([a,b])=L_p([a,b],\C, dx)$
 onto themselves, where $1\leq p\neq2$. These results were later substantially extended by J.~Lamperti\cite{La58}, who developed a general framework for studying isometries on function spaces defined via integral functionals. We refer interested readers to \cite{Fleming-1, Fleming-2} for a detailed exposition on this topic.

The study of embeddings between different $L_p$-spaces has also seen significant progress, fostering deep interactions between Banach space theory, probability theory, the geometry of convex bodies, harmonic analysis, and combinatorics. The pioneering works of P.~Lévy \cite{Le37} and I.J.~Schoenberg \cite{Sc38} established strong connections between positive definite functions, stable random variables, and isometric embeddings of Banach spaces into $L_p$-spaces. H.P.~Rosenthal~\cite{Rosenthal} obtained fundamental structural results for subspaces of $L_p([0,1])$ with $1 \le p < 2$.
Subsequently, F. Delbaen, H. Jarchow and A. Pe\l czy\'nski~\cite{Delbaen} characterized those subspaces of $L_p([0,1])$ that admit isometric embeddings into $\ell_p$ or finite-dimensional spaces $\ell_p^n$ for $0 < p < \infty$. It is also known that $L_q([0,1])\ie L_p([0,1])$ whenever $0 < p < q \le 2$ (see~\cite{Bret,Herz}). We refer interested readers to \cite{Johnson01,Pisier86} for a discussion of relevant topics.

 We also have the following result on the non-embeddability of a sequence space into a function space.

\begin{thm}(\!\!\cite[Theorem 2.1]{Dor_Israel})\label{Dor}
Let $1 < p < \infty$ and $\mu$ be a positive measure on $[0,1]$.
Then, $\ell_{q}^{2}(\R)\nie L_{p}([0, 1], \R, d\mu)$ for $q<p$ satisfying $(q,p)\in [1,2)\times(1,\infty)$ and for $q\neq p$ satisfying $(q,p)\in (2,\infty)\times(1,\infty)$.
\end{thm}
The above Theorem~\ref{Dor} was established in \cite{Dor_Israel} for $\mu$ equal to the Lebesgue measure; however, the same proof applies to any positive measure~$\mu$.

\subsection{Embeddability of Schatten classes.}

 In this subsection, we discuss affirmative and negative results for Problem \ref{prob-0} (2), (3).

\begin{thm}\label{simple_thm}
Let $m,n\in\N$ satisfy $2 \leq m \leq n$ and let $\dim(\hil)=\infty$.
Then,
\begin{align}
\label{simple_cases2}
&\ell_p^m(\C)\ie\cS_p^m,\quad \ell_p(\C)\ie\cS_p(\hil), \quad
\cS_p^m\ie\cS_p^n\ie\cS_p(\hil) \quad\text{for all }  p\in (0,\infty),\\
\label{simple_case3}& \ell_\infty^m(\C)\ie\cS_\infty^m,\quad c_0(\C)\ie\cS_\infty(\hil), \quad \cS_\infty^m\ie\cS_\infty^n\ie\cS_\infty(\hil),\\
\label{Schatten-2-case}
& \cS_2^m\ie \ell_2^{m^2}(\C)\ie\cS_p^{m^2} \quad\text{for all }  p\in (0,\infty],\\
\label{Schatten-2-case_infty}& \cS_2(\hil)\ie \ell_2(\C)\ie\cS_p(\hil) \quad\text{for all }  p\in (0,\infty].
\end{align}
\end{thm}

\begin{proof}
All embeddings in \eqref{simple_cases2} and \eqref{simple_case3} are immediate.

The embedding in \eqref{Schatten-2-case} was noted in \cite{ChHoPaPrRa22}  for $p=\infty$, and the same embedding map also works for every $p>0$, as observed later by S.K.~Ray. It can be justified as follows.

Note that $\cS_2^m\ie\ell_2^{m^2}(\C)$ via the identification
\[
[a_{ij}]_{m\times m}\mapsto(a_{ij})_{i,j=1}^m.
\]
Indeed, for $A=[a_{ij}]\in\cS_2^m$,
\[
\|A\|_2^2=\Tr(A^*A)=\sum_{i,j=1}^m |a_{ij}|^2 =\|(a_{ij})_{i,j=1}^m\|_2^2.
\]
Next, we demonstrate that the embedding $\ell_2^{m^2}(\C)\ie\cS_p^{m^2}$ for every $p\in (0,\infty]$ is realized by the map
\begin{align}\label{samya}
	(a_1,\ldots,a_{m^2})\mapsto
	\begin{bmatrix}
		a_1&a_2&\cdots&a_{m^2}\\
		0&0&\cdots&0\\
		\vdots&\vdots&\cdots&\vdots\\
		0&0&\cdots&0
	\end{bmatrix}.
\end{align}
Let $a=(a_1,\dots,a_{m^2})\in\ell_2^{m^2}(\C)$ and denote by $T(a)$
the matrix in \eqref{samya}. By a direct computation,
\[
|T(a)^*|=\big(T(a)T(a)^*\big)^{\frac12}={\rm diag}\Big(\Big(\sum\limits_{k=1}^{m^2}|a_k|^2\Big)^{\frac12},0,\dots,0\Big).
\]
Consequently,
\[\|T(a)\|=\|T(a)^*\|=\|\, |T(a)^*|\, \|=\|(a_1,\dots,a_{m^2})\|_2\]
and, for every $p\in (0,\infty)$,
\[
\|T(a)\|_p^p=\|T(a)^*\|_p^p= \Tr(|T(a)^*|^p)=\|(a_1,\dots,a_{m^2})\|_2^p.
\]
Thus, $T$ is an isometric embedding of $\ell_2^{m^2}(\C)$ into
$\cS_p^{m^2}$ for every $p\in (0,\infty]$.

Combining $\cS_2^m\ie\ell_2^{m^2}(\C)$ and $\ell_2^{m^2}(\C)\ie\cS_p^{m^2}$ yields \eqref{Schatten-2-case}.

\smallskip

An isometry for the inclusion \eqref{Schatten-2-case_infty} is constructed as follows. Let $f:\N\times\N\to \N$ be a bijective map and let $\{e_i\}_{i\in\N}$ be an orthonormal basis of $\hil$. The embedding $\cS_2(\hil)\ie\ell_2(\C)$ is implemented by the isometry $A\mapsto (a_{f(i,j)})_{i,j=1}^{\infty}$, where $a_{f(i,j)}=\langle A e_j, e_i\rangle$, and the embedding $\ell_2(\C)\ie \cS_p(\hil)$ by the isometry $(a_j)\mapsto B$, where $B$ is given by $\langle Be_j,e_i\rangle=\delta_{1i}\, a_j$. Combining these two embeddings yields \eqref{Schatten-2-case_infty}.
\end{proof}

The embedding $\cS_2^m\ie\cS_p^{m^2}$ ensured by \eqref{Schatten-2-case} contrasts with the non-embeddability $\ell_2^m(\mathbb K)\nie\ell_p^n(\mathbb K)$ given by Theorem~\ref{comm-thm-1} for any $2\le m\le n$ and $p\in2\N+1$.

For a wide range of parameters $(q,p)$, negative results concerning Problem~\ref{prob-0} (2), (3) are obtained in \cite{Gupta, Ra20, ChHoPaPrRa22, ChHoPrRa23}. In the sequel, we discuss the main ideas and sketch the proofs of these results.

\begin{thm}(\!\!\cite[Theorems 1.1 and 1.3]{Ra20}\label{thm:Ray})
Let $m \geq 2$ and $\dim(\hil)\geq 2$. Then, $$\ell_1^2(\C)\nie\cS_p(\hil),\quad \cS_1^m\nie\cS_p(\hil),\quad \cS_1(\hil)\nie\cS_p(\hil)\quad\text{for all }p \in (1,\infty].$$
\end{thm}

\begin{proof}[Sketch of the proof.]
	Suppose that there exists an isometry
	\[
	J_0 : \ell_1^2(\C)\ie\cS_\infty(\hil),
	\quad
	J_0(e_1)=A,\; J_0(e_2)=B ,
	\]
	where $e_1=(1,0)$ and $e_2=(0,1)$.
	It is shown in \cite{Ra20} that the existence of such an isometry, along with the unitary invariance of the operator norm and a Gram--Schmidt type orthogonalization argument, implies the existence of another isometry
	\[
	J_1 : \ell_1^2(\C) \ie \cS_\infty(\hil),
	\quad
	J_1(e_1)=\tilde A,\; J_1(e_2)=\tilde B ,
	\]
	where $\tilde A$ is a positive semidefinite diagonal operator.
Moreover, the analysis in \cite{Ra20} shows that all diagonal entries of $\tilde A$ are strictly less than~$1$, which contradicts the isometric condition $\|\tilde A\| =\|e_1\|_1=1$. This contradiction shows that $\ell_1^2(\C)\nie\cS_\infty(\hil)$.

Suppose that for a given $p\in (1,\infty)$,  there exists an isometry $J:\ell_1^2(\C)\ie\cS_p(\hil)$. Let $A=J(e_1)$ and $B=J(e_2)$. By properties of an isometry, $A$ and $B$ are linearly independent and $\|A\|_p=\|B\|_p=1$. Since $J$ is an isometry, for every $z\in\C$ with $|z|=1$,
\[\|A+zB\|_p=\|(1,z)\|_1=2=\|A\|_p+\|B\|_p.\]
Hence, by \cite[Proposition~4.1]{Bo19}, $A$ and $B$ are linearly dependent, which is a contradiction. Therefore, $\ell_1^2(\C)\nie\cS_p(\hil)$ for $1<p<\infty$.
	
By \eqref{simple_cases2}, $\ell_1^2(\C)\ie\cS_1(\hil)$ for an arbitrary Hilbert space $\hil$ with ${\rm dim}(\hil)\ge 2$. The latter along with $\ell_1^2(\C)\nie\cS_p(\hil)$ for $p \in (1,\infty]$ implies $\cS_1^m\nie\cS_p(\hil)$ and $\cS_1(\hil)\nie\cS_p(\hil)$.
\end{proof}

The result of Theorem \ref{thm:Ray} with $p=\infty$ under the additional assumption $\dim(\hil)<\infty$ was obtained in \cite{Gupta}.

The methods developed in \cite{Gupta,Ra20} are insufficient to study Problem \ref{prob-0} for general pairs
\[
(q,p)\in (0,\infty]\times (0,\infty].
\]
The general case was addressed in \cite{ChHoPaPrRa22, ChHoPrRa23}, where completely new techniques were introduced, in particular the application of the Kato--Rellich theorem and methods based on multilinear operator integration. The following results on non-embeddability are obtained in these two papers.

\begin{thm}(\!\!\cite[Theorem 1.1]{ChHoPaPrRa22})\label{thm:Chatto-1}
Let $2\leq m\leq n$. Then, $\cS_q^m\nie\cS_p^n$ for $q\neq p$ satisfying

\noindent
$(q,p)\in\big((1,\infty]\setminus\{2\}\times(1,\infty)\big)\bigcup
\big((1,\infty)\setminus\{2,3\}\times\{1\}\big) \bigcup \big((1,\infty)\setminus\{2\}\times\{\infty\}\big)$.
\end{thm}

\begin{thm}(\!\!\cite[Theorem 1.3]{ChHoPrRa23})\label{thm:Chatto-3}
	Let $2\leq m\leq n$.
	\begin{enumerate}
		\item $\cS_q^m\nie\cS_p^n$ for $(q,p)\in(0,2)\setminus\{1\}\times (0,1)$, $q\neq p$.
		\item Let $(q,p)\in [2,\infty)\times (0,1).$ Then there is no isometric embedding $T:S_q^m\to S_p^n$ with\\ $T(\rm{diag}(1,0,\ldots,0))=A,$ and $T(\rm{diag}(0,1,\ldots,0))=B$ such that
\begin{itemize}
\item  $A,B$ are $n\times n$ self-adjoint matrices, 
\item either $A\geq 0$ or $A\leq 0$.
\end{itemize}
\item $\cS_1^m\nie\cS_p^n$ for $p\in(0,1)\setminus\{\frac{1}{k}:k\in\mathbb{N}\}$.
\item $\cS_\infty^m\nie\cS_p^n$ for $p\in(0,1)\setminus\{\frac{1}{k}:k\in\mathbb{N}\}$.
\end{enumerate}
\end{thm}

\begin{thm}(\!\!\cite[Theorem 5.9]{ChHoPaPrRa22})\label{thm:Chatto-2}
 Let $\dim(\hil)=\infty$. Then, $\cS_q(\hil)\nie\cS_p(\hil)$ for $q\neq p$ satisfying $(q,p)\in\big((1,\infty)\setminus\{2\}\times[2,\infty)\big)\bigcup\big([4,\infty)\times\{1\}\big)
\bigcup\big(\{\infty\}\times (1,\infty)\big)\bigcup\big((2,\infty)\times\{\infty\}\big)$.
\end{thm}

Theorem \ref{thm:Chatto-2} is a particular case of Theorem \ref{thm-complete-new} proved below.

We outline the main ideas in the proofs of Theorems~ \ref{thm:Chatto-1} and \ref{thm:Chatto-3}, following \cite{ChHoPaPrRa22, ChHoPrRa23}. Firstly we recall the definition of a multilinear operator integral.

Let $C^k(\R)$ denote the space of $k$-times continuously differentiable functions, where $k\in\N\cup\{0\}$.  For $f\in C^k(\R)$, let $f^{[k]}$ denote the $k$th divided difference of $f$ defined recursively by $f^{[0]}=f$ and, for $k\in\N$,
\begin{align*}
f^{[k]}(\lambda_0,\dots,\lambda_k)
:=\begin{cases}\frac{f^{[k-1]}(\lambda_0,\dots,\lambda_{k-1})
-f^{[k-1]}(\lambda_0,\dots,\lambda_{k-2},\lambda_k)}{\lambda_{k-1}-\lambda_k} &\text{if }\; \lambda_{k-1}\neq\lambda_k,\\
\frac{\partial f^{[k-1]}}{\partial\lambda}(\lambda_0,\dots,\lambda_{k-2},\lambda)\big|_{\lambda=\lambda_{k-1}}
&\text{if }\; \lambda_{k-1}=\lambda_k.
\end{cases}
\end{align*}

Let $\dim(\hil)=m$, let $A_0,\ldots,A_n$ be self-adjoint operators on $\hil$, and let $\{\lambda_i^{(j)}\}_{i=1}^{d_j}$ be the distinct eigenvalues of $A_j$, with $d_j\le m$.
	Define the multilinear map $T^{A_0,\ldots,A_n}_{f^{[n]}}:\bh^n\to\bh$ by
	\begin{align}
\label{def-moi}
	&T^{A_0, A_1,\ldots,A_n}_{f^{[n]}}(B_1, B_2,\ldots,B_n)\\
\nonumber
&=\sum_{i_0=1}^{d_0}\cdots\sum_{i_n=1}^{d_n}
		f^{[n]}(\lambda_{i_0}^{(0)},\ldots,\lambda_{i_n}^{(n)})
		E_{A_0}(\{\lambda_{i_0}^{(0)}\}) B_1
 E_{A_1}(\{\lambda_{i_1}^{(1)}\}) B_2 \cdots B_n E_{A_n}(\{\lambda_{i_n}^{(n)}\}),
	\end{align}
	where $E_{A_j}$ denotes the spectral measure of $A_j$.
	The operator $T^{A_0,\ldots,A_n}_{f^{[n]}}$ is called the discrete multilinear operator integral with symbol $f^{[n]}$ (see \cite[Chapter~4]{Skripka-book}).

The following result follows from \cite[Theorem 5.3.2]{Skripka-book}.

\begin{thm}\label{thm:derivative}
Let $n\in\N$. Let $A,B$ be self-adjoint operators on $\ell^n_2$. Then for each $f\in C^2(\R)$, the function $\mathbb{R}\ni t\mapsto \Tr (f(A+tB))$ is twice differentiable, and its second derivative is given by
\[
\frac{d^2}{dt^2} \Tr (f(A+tB)) \Big|_{t=s}
=2!\Tr\!\left(T_{f^{[2]}}^{A+sB,A+sB,A+sB}(B,B)\right).
\]
\end{thm}

Theorem~\ref{thm:derivative} ensures the existence of the second derivative $\frac{d^2}{dt^2}\|A+tB\|_p^p\big|_{t=0}$ and provides an explicit representation of this derivative in terms of the eigenvalues of $A$ and the entries of $B$. If $p\geq 2$, the existence of $\frac{d^2}{dt^2}\|A+tB\|_p^p\big|_{t=0}$ follows from Theorem \ref{thm:derivative} applied to $f(x)=|x|^p$. If $p<2$ and $A$ is invertible, the existence of $\frac{d^2}{dt^2}\|A+tB\|_p^p\big|_{t=0}$ follows from Theorem \ref{thm:derivative} applied to $f\in C^2(\R)$ such that $f(x)=|x|^p$ away from $0$.
The representation for $\frac{d^2}{dt^2}\|A+tB\|_p^p\big|_{t=0}$ provided by Theorem~\ref{thm:derivative} helps determine whether the derivative is nonzero, as stated in the lemma below.

\begin{lma}(\!\!\cite[Lemma 4.1]{ChHoPaPrRa22}, \cite[Lemma 4.12]{ChHoPrRa23})\label{lem:nonzero derivative}
	Let $(q,p)\in (1,\infty)\times (0,\infty)\setminus\{1\}$ with $p\neq q$ and let $n\in\N$.  Let $A$ and $B$ be two self-adjoint operators on $\ell_2^n$ such that $A$ is diagonal and
	\[(1+|t|^q)^{p/q}= \|A+tB\|_p^p \quad \text{for all } t\in \R.\]
	\begin{enumerate}[(i)]
		\item Then for $p\geq 2$,
		\begin{align*}
			\frac{d^2}{dt^2}\norm{A+tB}_p^p\Big|_{t=0} > 0.
		\end{align*}
		\item If the map $\mathbb{R}\ni t\mapsto A+tB\in \mathcal{B}(\ell_2^n)$ is invertible in a neighborhood of $0$, then
\begin{align*}
\frac{d^2}{dt^2}\norm{A+tB}_p^p\Big|_{t=0}
\begin{cases} > 0 &\text{ for }\, p>1,\\  < 0 &\text{ for }\,  0<p<1 \text{ and }(A>0\text{ or }A<0).
\end{cases}
 \end{align*}	
\end{enumerate}
\end{lma}

\begin{proof}[\bf Sketch of the proofs of Theorems~\ref{thm:Chatto-1}, \ref{thm:Chatto-3}, and \ref{thm:Chatto-2}]  We note that $q\neq p$ in all cases.

We first outline the proofs of Theorems~\ref{thm:Chatto-1} and \ref{thm:Chatto-3}, where the dimension of the underlying Hilbert spaces is finite. In each case, the proof of $\cS_q^m\nie\cS_p^n$ proceeds by contradiction.
	
	Suppose that $\cS_q^m$ isometrically embeds into $\cS_p^{n}$. Since $\ell_q^2(\C)$ embeds isometrically into $\cS_q^m$, the composition yields an isometric embedding
\[T:\ell_q^2(\C)\ie \cS_p^n.\]
For $p,q<\infty$, consider the linear map $J: \ell_q^2(\C)\ie\cS_p(\C^n\oplus\C^n)$ given by
\[J((a_1, a_2))=\frac{1}{2^{1/p}}
\begin{bmatrix}
0 & a_1T(e_1) + a_2 T(e_2)\\
a_1T(e_1)^* + a_2 T(e_2)^* & 0
\end{bmatrix}\quad \text{ for }(a_1,a_2)\in \C^2.
\]
Note that $$A:=J(e_1)=J((1,0)),\quad B:=J(e_2)=J((0,1))$$ are self-adjoint and $$|J((a_1,a_2))|^2= 2^{-2/p} \, {\rm diag} \big(|a_1T(e_1)^*+a_2T(e_2)^*|^2, |a_1T(e_1)+a_2T(e_2)|^2\big).$$ Therefore, since $T$ is an isometry, we have
\begin{align*}
	\|J((a_1, a_2))\|_p^p = &\frac{1}{2}(\|a_1T(e_1)^*+a_2T(e_2)^*\|_p^p + \|a_1T(e_1)+a_2T(e_2)\|_p^p)\\
	=&\frac{1}{2}(\|\bar{a}_1 T(e_1) + \bar{a}_2 T(e_2)\|_p^p + \|a_1T(e_1)+a_2T(e_2)\|_p^p)\\
	=&\frac{1}{2}(\|T((\bar{a}_1, \bar{a}_2))\|_p^p + \|T((a_1, a_2))\|_p^p)\\
	=& \|(a_1, a_2)\|_q^p.
\end{align*}
Hence, $J:\ell_q^2(\C)\to \cS_p(\C^n\oplus\C^n)$ is an isometry and, consequently,
	\begin{align}\label{eq:scal-1}
		(1+|t|^q)^{p/q}=\|(1,t)\|_q^p = \|A+tB\|_p^p \quad \forall t\in\R.
	\end{align}
	
	Since Schatten $p$-norms are unitarily invariant, we may further assume that $A$ is diagonal and $B$ is self-adjoint in \eqref{eq:scal-1}. Let $\{\lambda_i(t)\}_{i=1}^{2n}$ denote the eigenvalues of $A+tB$ (counted with multiplicity). Then \eqref{eq:scal-1} reduces to
	\begin{align}\label{eq:scal-2}
		(1+|t|^q)^{p/q}= \sum_{i=1}^{2n} |\lambda_i(t)|^p \quad \forall t\in\R.
	\end{align}
For $p=\infty$, the same construction is used with $2^{-1/p}=1$, and \eqref{eq:scal-1} and \eqref{eq:scal-2} are replaced by $\|A+tB\|=\|(1,t)\|_q=\max_i|\lambda_i(t)|$. For $q=\infty$ and $p<\infty$, their left-hand sides are replaced by $\max\{1,|t|^p\}$.
	
	The remainder of the argument compares the analytic behavior of the scalar functions on both sides of \eqref{eq:scal-2}. The left-hand side is explicit, while the right-hand side is more delicate, since eigenvalues of linear perturbations of diagonal matrices are not necessarily linear in the perturbation parameter. Hence, the analytic behavior of the right-hand side of \eqref{eq:scal-2} is not readily determined.
To address this issue, \cite{ChHoPaPrRa22} introduces the use of the Kato--Rellich theorem, which asserts that the eigenvalues $\{\lambda_i(t)\}_{i=1}^{2n}$ are locally analytic functions of $t$. Below we briefly outline the arguments for the non-embeddability $\cS_q^m\nie\cS_p^{ n}$, case by case.
\smallskip

{\it Case 1:\;} $(q,p)\in (1,\infty)\setminus\{2,3\}\times\{1\}$.

\noindent
For $1<q<4$, the identity \eqref{eq:scal-2} would, via the Kato--Rellich theorem, ultimately imply that $(1+|t|^q)^{1/q}$ admits an analytic extension near $0$, which is false.
	For $q\geq 4$, \eqref{eq:scal-1} and $2$-uniform PL-convexity of $\cS_1^{2n}$ \cite[Theorem 4.3]{DaGaTo84} would imply that there exists $c\geq 1/2$ such that $(1+c)^q\leq 4$, which is impossible.
\smallskip	

{\it Case 2:\;} $(q,p)\in (1,\infty)\setminus\{2\}\times\{\infty\}$.

\noindent
For $1<q<2$, the argument is analogous to that in {\it Case 1}. For $q>2$, let
\[\mathcal{E}={\rm span}\{x \in \ell_2^{2n}: Ax=x \text{ or } Ax=-x\},\]
and let $P$ be the orthogonal projection onto $\mathcal{E}$. Note that $(I-P)AP=PA(I-P)=0$ and $\|(I-P)A(I-P)\|<1$. For a unit vector $x\in \mathcal{E}$, \eqref{eq:scal-1} implies
\[\|(A+tB)x\|\le \|A+tB\|=(1+|t|^q)^{1/q},  \quad t\in\R.	\]
For a demonstration of the argument, assume that $Ax=x$ and $\la x,Bx\ra\geq 0$. Consequently,
\begin{align*}	
 (1+\|Bx\|^2t^2)^q &  \leq\bigl(1+2t\la x,Bx\ra+\|Bx\|^2t^2\bigr)^q\\
&=\bigl(1+t(\la Ax,Bx\ra+\la Bx,Ax\ra)+t^2\|Bx\|^2\bigr)^q=\|(A+tB)x\|^{2q}\\
&\le 1+2|t|^q+|t|^{2q}, \quad  t>0.
\end{align*}
Substituting $t$ with $t^{-1}$ in the preceding inequality and applying the binomial power series expansion yields
\begin{align*}
&1+ q\|Bx\|^2t^{-2} + \mathcal{O}(\|Bx\|^4t^{-4}) \leq 1+2t^{-q}+t^{-2q}, \quad t\gg 1,
\end{align*}
which implies that $\|Bx\|=0$ and, hence, $B=(I-P)B(I-P)$.
Noticing the block matrix representation
\begin{align*}
A+tB=\begin{bmatrix}PAP & 0 \\0 & (I-P)(A+tB)(I-P)\end{bmatrix},
\end{align*}
we obtain $\|A+tB\|=\max\bigl\{\|PAP\|,\ \|(I-P)(A+tB)(I-P)\|\bigr\}\le 1$ for sufficiently small $t$, which contradicts \eqref{eq:scal-1}.
Thus $\cS_q^m\nie \cS_\infty^n$.
	
{\it Case 3:\;} $(q,p)\in (0,2)\setminus\{1\}\times(0,1)$.

\noindent
Noncommutative Clarkson's inequalities \cite[Theorem 2.7]{Mc67} (see also the proof of Theorem~\ref{thm-clarkson-infinite}) applied to \eqref{eq:scal-1} imply $p\leq q$. Since $p\neq q$, it follows that $p<q$. To obtain a contradiction, one uses the Kato--Rellich theorem and expands the left-hand side of \eqref{eq:scal-2} via a binomial power series.
Comparing the exponents and the corresponding coefficients in the expansion of $\sum_{i=1}^{2n}|\lambda_i(t)|^p$ with those in the binomial expansion around $t=0$ produces the desired contradiction.
\smallskip

{\it Case 4:\;} $(q,p)\in \big(\{\infty\}\times(0,\infty)\setminus\{1/k:k\in\N\}\big)\bigcup \big(\{1\}\times(0,1)\setminus\{1/k:k\in\N\}\big)$.

\noindent
If $q=1$, we apply the Kato--Rellich theorem and expand the left-hand side of \eqref{eq:scal-2} as a binomial power series around $t=0$.  Comparing the exponents and the corresponding coefficients in the expansion of $\sum_{i=1}^{2n}|\lambda_i(t)|^p$ with those in the binomial expansion yields the desired contradiction.

If $q=\infty$ and $0<p<1$, we have $\max\{1,|t|^p\}=\sum_{i=1}^{2n}|\lambda_i(t)|^p$ for $t\in\R$. By the Kato--Rellich theorem, the eigenvalues  $\lambda_i(t)$ can be chosen to have convergent power series expansions around $t=1$. Comparing the exponents and the corresponding coefficients in the expansions on both sides of the above equality yields the desired contradiction.

Next, assume $(q,p)\in \{\infty\}\times (1, \infty)$. In this case, the left-hand side of $\max\{1,|t|^p\}=\|A+tB\|_p^p$ (a suitable modification of \eqref{eq:scal-1}) is not differentiable at $t=1$, whereas, by \cite[Lemma 5.8.4]{Skripka-book}, its right-hand side is differentiable at $t=1$.
	
The above contradictions completes the proof of {\it Case 4}.

\smallskip

{\it Case 5:\;} $(q,p)\in(1,\infty)\setminus\{2\}\times[2,\infty)$.

\noindent
Expanding the left-hand side of \eqref{eq:scal-2} via a binomial power series, we obtain
   \begin{align}\label{eq:scal-3}
   	1+ \frac{p}{q} |t|^q + \mathcal{O}(|t|^{2q}) = \sum_{i=1}^{2n} |\lambda_i(t)|^p \quad \text{for all } t\in (-1,1),
   \end{align}
By Theorem \ref{thm:derivative}, the map
$t \mapsto \|A+tB\|_p^p$ is twice differentiable at $0$. Hence, we obtain
\begin{align}\label{eq:scal-3aa}
\|A+tB\|_p^p=\|A\|_p^p + t\, \frac{d}{dt}\|A+tB\|_p^p\big|_{t=0}
	+\frac{t^2}{2!}\frac{d^2}{dt^2}\|A+tB\|_p^p\big|_{t=0}	+ o(|t|^2).
\end{align}
Combining \eqref{eq:scal-1}, \eqref{eq:scal-3}, and \eqref{eq:scal-3aa} yields
\begin{align}\label{eq:scal-3b}
	1+\frac{p}{q}|t|^q+\mathcal{O}(|t|^{2q})
	= 1 + t\, \frac{d}{dt}\|A+tB\|_p^p\big|_{t=0}
	+\frac{t^2}{2!}\frac{d^2}{dt^2}\|A+tB\|_p^p\big|_{t=0}
	+ o(|t|^2).
\end{align}
Next, note that \eqref{eq:scal-1} implies that $\|A+tB\|_p^p$ attains its minimum at $t=0$, and hence
\[
\frac{d}{dt}\|A+tB\|_p^p\big|_{t=0}=0.
\]
Consequently, \eqref{eq:scal-3b} reduces to
\begin{align}\label{eq:scal-4}
	\lim_{t\to0^+} 2\,\frac{p}{q}\, t^{q-2}
	= \frac{d^2}{dt^2}\|A+tB\|_p^p\Big|_{t=0}.
\end{align}
By Lemma \ref{lem:nonzero derivative}~(i), the right-hand side of \eqref{eq:scal-4} is nonzero, which forces $q=2$. This contradicts the assumption that $(q,p)\in(1,\infty)\setminus\{2\}\times[2,\infty)$.

\smallskip	
	
{\it Case 6:\;} $(q,p)\in (1,\infty)\setminus\{2\}\times(1,2)$.

\noindent
If $A$ is invertible, then $t\mapsto A+tB$ is invertible in a neighborhood of $0$.
Therefore, by Lemma~\ref{lem:nonzero derivative}~(ii), $\frac{d^2}{dt^2}\norm{A+tB}_p^p\big|_{t=0}\neq 0$. Combining the latter with \eqref{eq:scal-4} implies $q=2$, contradicting the assumption that $q\neq 2$.
		
If $A$ is singular, consider the invertible linear perturbation $A_x = A + xI$ of $A$ for $x$ in a sufficiently small interval $(0,\epsilon)$.  Applying the identity \eqref{eq:scal-1} along with the Kato-Rellich theorem and noting that the eigenvalues of $A_x+tB$ are $\lambda_i(t)+x$, one can derive that
\[\frac{d^2}{dt^2}\|A+tB\|_p^p\Big|_{t=0} = \lim_{x\to 0^+} \frac{d^2}{dt^2}\|A_x+tB\|_p^p\Big|_{t=0}.\]
By \cite[Remark 4.3]{ChHoPaPrRa22}, the latter derivative is bounded below by $p(p-1)>0$. Combining this with \eqref{eq:scal-4} yields $q=2$,
giving the same contradiction as in the invertible case.

\smallskip

{\it Case 7:\;} $(q,p)\in [2,\infty)\times(0,1)$, either $A\geq 0$ or $A\leq 0$.

\noindent
Suppose there exists an isometric embedding $T:\cS_q^m\to \cS_p^n$ with $T(\text{diag}(1,0,\ldots,0))=A$ and $T(\rm{diag}(0,1,\ldots,0))=B$ such that $A=A^*$, $B=B^*$, and either $A\geq0$ or $A\leq 0$. Therefore,
\begin{align*}
	(1+|t|^q)^{p/q} = \|{\rm diag}(1,t,\ldots,0)\|_q^p = \|T({\rm diag}(1,t,\ldots,0))\|_p^p = \|A+tB\|_p^p\quad \text{for all } t\in\R,
\end{align*}
leading to \eqref{eq:scal-4}.
If $A$ is invertible, then Lemma~\ref{lem:nonzero derivative}~(ii) gives
$\frac{d^2}{dt^2}\|A+tB\|_p^p\big|_{t=0}<0$, contradicting \eqref{eq:scal-4}, whose left-hand side is nonnegative. If $A$ is singular, an analysis of the eigenvalue expansions yields the same contradiction (see the proof of\cite[Theorem 4.13]{ChHoPrRa23}).
	
\smallskip

Finally, Theorem~\ref{thm:Chatto-2}, with $\dim(\hil)=\infty$, is obtained by adapting the arguments underlying Theorem~\ref{thm:Chatto-1}, specifically those in {\it Cases~1, 2, 4}, and~{\it 5}. The same arguments as in {\it Cases~1, 2}, and~{\it 4} apply for $(q,p)\in[4,\infty)\times\{1\}$, $(q,p)\in(2,\infty)\times\{\infty\}$, and $(q,p)\in\{\infty\}\times(1,\infty)$, respectively. For {\it Case~5} with $p>2$, one uses additional approximation results based on multilinear operator integration from
\cite[Theorem 18]{PoSu14} (see also \cite[Theorem 36]{PSTZ}), along with the analogue of Theorem~\ref{thm:derivative} established in \cite[Theorem 16]{PoSu14} (see also \cite[Theorem 42]{PSTZ}). The case $p=2$ follows directly from the Hilbert space structure of $\cS_2(\hil)$. For details, see \cite[proof of Theorem 5.9]{ChHoPaPrRa22}.
\end{proof}

By utilizing the $2$-uniform PL-convexity of $\cS_1(\hil)$ similarly to how it was done in the proof of Theorem \ref{thm:Chatto-1} in Case 1 above, we obtain the following non-embeddability result.

\begin{thm}\label{thmrmrk}
Let $\dim(\hil) \geq 2$. Then,
\[\cS_\infty(\hil)\nie \cS_1(\hil).\]
\end{thm}

Since the Kato-Rellich theorem is not available in its full strength for infinite-dimensional operators, different approaches are taken to treat the case $p<2$. One of the approaches employs Theorem \ref{Dor} and a corollary of the recent groundbreaking work \cite{Otte_Inv} stated below.

\begin{thm}(\!\!\cite[Corollary 1.4]{Otte_Inv}\label{Otte})
 Let $\dim(\hil)=\infty$, $q>0$, $A, B \in \cS_q(\hil)$, and let
 ${\rm span}_\R\{A,B\}_{\cS_q}$ denote the $\|~ \|_q$-span of $A$ and $B$ over $\R$. Then, there is a positive measure $\mu$ on $[0,1]$ such that
$${\rm span}_\R\{A,B\}_{\cS_q}\ie L_q([0, 1], \R, d\mu).$$
\end{thm}

\begin{thm}\label{mainthm}
 Let $\dim(\hil)=\infty$. Then,  $\cS_q(\hil)\nie\cS_p(\hil)$ for $q<p$ satisfying $(q,p)\in [1,2)\times(1,\infty)$ and for $q\neq p$ satisfying $(q,p)\in(2,\infty)\times (1,\infty)$.
\end{thm}

\begin{proof}
Note that $\ell_{q}^{2}(\R)\ie {\rm span}_\R\{A,B\}_{\cS_q}\ie \cS_{q}(\hil)$,  where the first embedding is given by the isometry mapping $(1,0)\mapsto A=\rm{diag}(1,0,0,\cdots)$ and $(0,1)\mapsto B= \rm{diag}(0,1,0,\cdots)$.
If $\cS_{q}(\hil)$ embeds into $\cS_{p}(\hil)$, then it follows that $\ell_{q}^{2}(\R)$ embeds into $\cS_{p}(\hil)$. By Theorem~\ref{Otte}, $\ell_{q}^{2}(\R)$ embeds into $L_{p}([0, 1], \R, d\mu)$. Finally, an application of Theorem \ref{Dor} concludes the proof by contradiction.
\end{proof}

Theorem \ref{mainthm} is a particular case of Theorem \ref{thm-complete-new} proved below.
	
	\smallskip
	
	The argument leading to \cite[Lemma 2.2(1)]{ChHoPrRa23} also
	applies to Schatten classes on infinite dimensional Hilbert
	spaces and gives the following result.
	
\begin{thm}\label{thm-clarkson-infinite}
Let $ \dim(\hil)\geq 2$. Then,
\[\ell_q^2(\C)\nie\cS_p(\hil),\quad\cS_q(\hil)\nie\cS_p(\hil)\quad\text{for all } 0<q<p\leq 2.\]
\end{thm}
	
	\begin{proof}
		Suppose that there exists a linear isometric embedding
		\[T:\ell_q^2(\C)\ie\cS_p(\hil).\]
		Set $A=T((1,0))$ and $B=T((0,1))$. By linearity and isometry,
		\[\|A\|_p=\|B\|_p=1,\quad \|A+B\|_p^p=\|A-B\|_p^p=2^{p/q}.\]
		For $0<p\leq 2$, the noncommutative Clarkson inequality \cite[Theorem 2.7]{Mc67} gives
		\[\|A+B\|_p^p+\|A-B\|_p^p\leq 2(\|A\|_p^p+\|B\|_p^p).\]
		Hence,
		\[2\cdot 2^{p/q}=\|A+B\|_p^p+\|A-B\|_p^p\leq 2(1+1)=4.\]
		It follows that $p\leq q$, contradicting $q<p$.
		
		As in \eqref{simple_cases2}, $\ell_q^2(\C)\ie\cS_q^2\ie\cS_q(\hil)$.
		Hence, an isometric embedding of $\cS_q(\hil)$ into $\cS_p(\hil)$ would yield an isometric embedding of $\ell_q^2(\C)$ into $\cS_p(\hil)$, contradicting the
		first conclusion.
	\end{proof}
	
	The following result was suggested by the referee. It incorporates the results of Theorems~\ref{thm:Chatto-2} and~\ref{mainthm}.

\begin{thm}\label{thm-complete-new}
Let $\dim(\hil)=\infty$. Then,
\begin{align*}\cS_q(\hil)\nie\cS_p(\hil)
\quad\text{for all }q\neq p\; \text{satisfying }(q,p)\in[1,\infty]\setminus\{2\} \times[1,\infty].
\end{align*}
\end{thm}

\begin{proof}
	Suppose $\cS_q(\hil)\ie\cS_p(\hil)$ and for this proof adopt the convention $\ell_\infty(\C):=c_0(\C)$. Then, by Theorem \ref{simple_thm}, $\ell_{q}(\C)\ie \cS_p(\hil)$. Let $\phi: \ell_{q}(\C) \to \cS_p(\hil)$ be an isometry, which means $\phi: \ell_{q}(\C) \to \phi(\ell_q(\C))$ is an isomorphism. If $p=2$, then $Y:=\phi(\ell_q(\C))$ is a closed infinite dimensional subspace of the separable Hilbert space
	$\cS_2(\hil)$, and hence is isomorphic to $\ell_2(\C)$. If $p\neq2$, then by \cite[Proposition 4]{Ar} (see also \cite[Theorems 1, 2]{Fr}, \cite[Theorem 1]{Ho}), $\phi(\ell_q(\C))$ contains a subspace $Y$ that is isomorphic to either $\ell_2(\C)$ or $\ell_p(\C)$. Consequently, the subspace $Z := \phi^{-1}(Y) \subset \ell_{q}(\C)$ is also isomorphic to either $\ell_2(\C)$ or $\ell_p(\C)$, which is impossible by \cite[Corollary 2.1.6]{KaltonBook} since $p \neq q \neq 2$. Therefore, $\cS_q(\hil)\nie\cS_p(\hil)$.
\end{proof}	

 It follows from the proof of Theorem \ref{thm-complete-new} that $\cS_q(\hil)$ does not embed isomorphically into $\cS_p(\hil)$ for all $q\neq p$ satisfying $(q,p)\in[1,\infty]\setminus\{2\} \times[1,\infty]$.

\section{Open problems}	
The methods and arguments presented in the preceding section do not suffice to resolve Problem~\ref{prob-0} completely; several questions therefore remain open and merit further investigation.

\begin{prob}\label{Pr}
	Let $m,n\in\N$ with $2\le m\le n$. Let $p, q\in(0,\infty]$ and $p\neq q$. The following questions concerning isometric embeddings of Schatten classes remain unresolved:
	\begin{enumerate}[(i)]
		\item Does $\cS_2^m$ embed isometrically into $\cS_p^n$ for $p\in(0,\infty]\setminus\{2\}$ and $n<m^2$?
		
		This question remains unresolved in Theorem \ref{simple_thm}.
		\smallskip
		
		\item Does $\cS_3^m$ embed isometrically into $\cS_1^n$?
		
		This question remains unresolved in Theorem \ref{thm:Chatto-1}.
		
		\smallskip
		
		\item Do $\cS_1^m$ and $\cS_\infty^m$ embed isometrically into $\cS_p^n$ for $p\in\{\frac{1}{k}:k\in\mathbb{N}{\setminus \{1\}}\}$?
		
		This question remains unresolved in Theorem \ref{thm:Chatto-3} (3),(4).
		
		\smallskip
		
		\item Does $\cS_q^m$ embed isometrically into $\cS_p^n$ for $(q, p)\in (2, \infty)\times (0,1)$?
		
		This question remains unresolved in Theorem \ref{thm:Chatto-3} (2).
		\smallskip
		
		\item Does $\cS_q(\hil)$ embed isometrically into $\cS_p(\hil)$ when $\dim(\hil)=\infty$ for
		
		$(q,p)\in (0,\infty]\setminus\{2\}\times (0,1)$ with $q>p$?
		
		This question remains unresolved in  Theorems \ref{simple_thm}, \ref{thmrmrk}, \ref{thm-clarkson-infinite}, and \ref{thm-complete-new}.
	\end{enumerate}
\end{prob}

\section*{Acknowledgements}

A.~Chattopadhyay and C.~Pradhan thank Samya K.~Ray for introducing the problem of isometric embeddings of Schatten classes and express their deepest gratitude to Otte Hein\"avaara for explaining the details of his result \cite[Corollary 1.4]{Otte_Inv} and for bringing to attention \cite[Theorem 2.1]{Dor_Israel}, the result of which had been independently obtained without prior knowledge of the work. A. Chattopadhyay is supported by the Core Research Grant (CRG), File No: CRG/2023/004826, of SERB. A. Skripka is supported in part by Simons Foundation Grant MP-TSM-00002648. C. Pradhan acknowledges support from the Fulbright-Nehru postdoctoral fellowship, Award No. 3130/FNPDR/2025. The authors thank the anonymous referee for several valuable suggestions, including the proof of Theorem \ref{thm-complete-new}, which greatly improved the exposition of this article.

\end{document}